\documentclass{amsart}

\usepackage{amssymb,amsmath,graphicx}

\usepackage{xcolor}

\newtoks\prt

\numberwithin{equation}{section}

\newtheorem{thm}{Theorem}[section]

\newtheorem{lemma}[thm]{Lemma}

\newtheorem{prop}[thm]{Proposition}

\newtheorem{cor}[thm]{Corollary}

\theoremstyle{definition}

\def\eqn#1$$#2$${\begin{equation}\label#1#2\end{equation}}

\def\C{\mathcal C}

\def\E{\mathcal E}

\def\e{e^\ast}

\def\E{E^\ast}

\def\ep{\varepsilon}

\def\en{\mathbb N}

\def\er{\mathbb R}

\def \reg {\partial _{\kern1pt\text{reg}}}

\def\la{\langle}

\def\ra{\rangle}

\newcommand{\norm}[1]{\left\|#1\right\|}

\newcommand{\abs}[1]{\left| #1  \right|}

\renewcommand\qedsymbol{$\blacksquare$}

\begin{document}

	\title[Isomorphisms of $\C(K, E)$ spaces and height of $K$]
	{Isomorphisms of $\C(K, E)$ spaces and height of $K$}

	\author[J.~Rondo\v s]{Jakub Rondo\v s}
	\address[J.~Rondo\v s]{Charles University\\
		of Mathematics and Physics\\
		Department of Mathematical Analysis \\
		Sokolovsk\'{a} 83, 186 \ 75\\Praha 8, Czech Republic}
	
	\email{jakub.rondos@gmail.com}
	
	\author[J.~Somaglia]{Jacopo Somaglia}
	\address[J.~Somaglia]{Dipartimento di Scienze economiche e Aziendali\\  \\Universit\`a degli Studi di Pavia\\Via San Felice 5,
		27100 Pavia, Italy}
	\email{jacopo.somaglia@unipv.it}


	
	\thanks{}

	\begin{abstract}
		Let $K_1$, $K_2$ be compact Hausdorff spaces and $E_1, E_2$ be Banach spaces not containing a copy of $c_0$. We establish lower estimates of the Banach-Mazur distance between the spaces of continuous functions $\C(K_1, E_1)$ and $\C(K_2, E_2)$ based on the ordinals $ht(K_1)$, $ht(K_2)$, which are new even for the case of spaces of real valued functions on ordinal intervals. As a corollary we deduce that $\C(K_1, E_1)$ and $\C(K_2, E_2)$ are not isomorphic if $ht(K_1)$ is substantially different from $ht(K_2)$.
	\end{abstract}
	
	\maketitle

	
	\section{Introduction}

	We first recall several notions. To start with, for a compact space $K$ and a Banach space $E$, let $\C(K,E)$ denote the space of all continuous $E$-valued functions endowed with the sup-norm. We write $\C(K)$ for $\C(K, \er)$. All compact spaces are assumed to be Hausdorff.
	
	Next we recall that the \emph{Banach-Mazur distance} of Banach spaces $E_1$, $E_2$ is defined to be the infimum of $\norm{T}\norm{T^{-1}}$ over the set of all isomorphisms $T:E_1 \rightarrow E_2$ and is denoted by $d_{BM}(E_1, E_2)$.
	
	Further, the \emph{derivative} of a topological space $S$ is defined recursively as follows. The set $S^{(1)}$ is the set of accumulation points of $S$, and for an ordinal $\alpha>1$, let $S^{(\alpha)}=(S^{\beta})^{(1)}$, if $\alpha=\beta+1$, and $S^{(\alpha)}=\bigcap_{\beta<\alpha} S^{(\beta)}$, if $\alpha$ is a limit ordinal. Moreover, let $S^{(0)}=S$. The topological space $S$ is called \emph{scattered} if there exists an ordinal $\alpha$ such that $S^{(\alpha)}$ is empty, and minimal such $\alpha$ is called the \emph{height} of $S$ (or the Cantor-Bendixon index of $S$) and is denoted by $ht(S)$. If $S$ is not scattered, then we define $ht(S)$ to be $\infty$.  We use the convention that $\alpha< \infty$ for each ordinal $\alpha$.
	If $K$ is a scattered compact space, then $ht(K)$ is always a successor ordinal.
	
	Further, throughout the paper, the symbol $\en$ will stand for the set of positive integers. Next we recall that a gamma number is an ordinal which is greater than the sum of any two lesser ordinals. Equivalently, gamma numbers are either $0$ or ordinals of the form $\omega^{\alpha}$ for some ordinal $\alpha$, see \cite{monk_set_theory}. Given an ordinal $\alpha$, we let $\Gamma(\alpha)$ denote the minimum gamma number which is not less than $\alpha$. Since $\omega^{\alpha} \geq \alpha$ for any ordinal $\alpha$, this minimum exists. For completeness, let $\Gamma(\infty)=\infty$.
	
	In this paper we study the isomorphism classes and distances between $\C(K, E)$ spaces. Let us begin with recalling known results in this area. 
	
	First we recall that the classical isomorphic classification due to Bessaga and Pe\l czy\'{n}ski \cite{BessagaPelcynski_classification} together with the Milutin theorem \cite[Theorem 2.1]{RosenthalC(K)} can be formulated in the way that for two compact metric spaces $K_1, K_2$, the Banach spaces of continuous functions $\C(K_1)$ and $\C(K_2)$ are isomorphic if and only if $\Gamma(ht(K_1))=\Gamma(ht(K_2))$. In the case of nonmetrizable compact spaces, such a classification is no longer true even for the case of ordinal intervals, since there other properties of compact spaces that are preserved by isomorphisms of the corresponding spaces of continuous functions, e.g., cardinality (see \cite{cengiz}). For example, $\Gamma(ht([0, \omega_1)])=\Gamma(ht([0, \omega_1 2]))$, but the space $\C([0, \omega_1])$ is not isomorphic to $\C([0, \omega_1 2])$, see \cite{Semadeni196081}. On the other hand, the condition $\Gamma(ht(K_1))=\Gamma(ht(K_2))$ still remains necessary for the spaces $\C(K_1)$ and $\C(K_2)$ to be isomorphic, which follows from the computation of the Szlenk index of a $\C(K)$ space due to Causey \cite{CAUSEY_C(K)_index}, and the fact that the Szlenk index is preserved by isomorphisms of Banach spaces.
	
	Now we turn our attention to case of spaces of vector-valued functions. 
	The isomorphic classification of $\C(K, E)$ spaces for certain Banach spaces $E$ and metrizable compact spaces $K$ has been treated in \cite{Galego_vector_ordinals}, \cite{GALEGO_vector_classification_factor}, and \cite{GALEGO_Zahn_classification}. In \cite{GALEGO_Zahn_classification}, the authors prove that if a Banach space $E$ has some uniformly convex $\omega_1$-quotient (see \cite[Definition 2.1]{GALEGO_Zahn_classification}), in particular, if $E$ is uniformly convex (see \cite[Remark 2.3]{GALEGO_Zahn_classification}), then for all infinite compact metric spaces $K_1$ and $K_2$, $\C(K_1, E)$ is isomorphic to $\C(K_2, E)$ if and only if $\C(K_1)$ is isomorphic to $\C(K_2)$. In the area of nonmetrizable compact spaces, there is much less known. It was proved in \cite[Theorem 1.4]{CandidoScattered} that if $E$ is a Banach space not containing an isomorphic copy of $c_0$, $K_2$ is a scattered compact space and $\C(K_1)$ embeds isomorphically into $\C(K_2, E)$, then $K_1$ is also scattered. Also, the Szlenk index of a general $\C(K, E)$ space has been computed in \cite[Theorem 1.4]{Causey_szlenk_hulls}, it is simply the maximum of Szlenk indices of $\C(K)$ and $E$. This important result, however, clearly cannot be used to deduce that the compact spaces $K_1$, $K_2$ have similar properties whenever the spaces $\C(K_1, E), \C(K_2, E)$ are isomorphic, in the case when $Sz(E) \geq \max \{Sz(\C(K_1)), Sz(\C(K_2))\}$, in particular, if $E$ is not an Asplund space.
	
	Now, let us recall the known lower estimates of the Banach-Mazur distance between a couple of $\C(K, E)$ spaces. We begin with the classical Amir-Cambern theorem, which asserts that, for nonhomeomorphic compact spaces $K_1$ and $K_2$, the distance between $\C(K_1)$ and $\C(K_2)$ is at least $2$. Vector-valued extensions of this result were subsequently proven by several authors, see e.g. the important paper \cite{cidralgalegovillamizar}. Other than that, there have been proved better estimates based on the properties of derived sets of the compact spaces. Firstly, in \cite{Gordon3}, Gordon proved that if $K_1$, $K_2$ are compact spaces such that $d_{BM}(\C(K_1), \C(K_2))<3$, then all derivatives of $K_1$ and $K_2$ have the same cardinality. This result was generalized to the case of vector-valued functions in \cite[Theorem 1.5]{CandidoGalego3} and \cite[Theorem 1.7]{GalegoVillamizar}. 
	Next, it was proved in \cite[Theorem 1.2]{CandidoGalegoComega} that if $K$ is a compact space with $K^{(n)}$ nonempty for some $n \in \en$ and $F$ is a compact space with $F^{(2)}=\emptyset$, then $d_{BM}(\C(K), \C(F)) \geq 2n-1$.  Moreover, if $\abs{K^{(n)}}>\abs{F^{(1)}}$, then $d_{BM}(\C(K), \C(F)) \geq 2n+1$. In \cite[Theorem 1.1]{CANDIDOc0} it has been showed that if $\Gamma$ is an infinite discrete space, $E$ is a Banach space not containing an isomorphic copy of $c_0$ and $T: \C(K) \rightarrow \C_0(\Gamma, E)$ is an into isomorphism, then for each $n \in \en$, if $K^{(n)}$ is nonempty, then $\norm{T}\norm{T^{-1}} \geq 2n+1$. 
	Similar results for isomorphisms with range in $\C_0(\Gamma, E)$ spaces were proven before in \cite{CandidoGalegoFund} and \cite{CandidoGalegoCotype}. These estimates were extended in \cite{rondos-scattered-subspaces} to the case of two spaces $K_1$ and $K_2$ of finite height.
	
	The purpose of this paper is to show that for Banach spaces $E_1$ and $E_2$ not containing an isomorphic copy of $c_0$, the condition $\Gamma(ht(K_1))=\Gamma(ht(K_2))$ remains necessary for the spaces $\C(K_1, E_1)$ and $\C(K_2, E_2)$ to be isomorphic, and also, that the known distance estimates between spaces of continuous functions on compact spaces of finite height can be extended to compact spaces of arbitrary height. More precisely, we have the following result.
	
	\begin{thm}
		\label{theresult}
		Let $K_1, K_2$ be infinite compact spaces, $E$ be a Banach space containing no copy of $c_0$ and $T:\C(K_1) \rightarrow \C(K_2, E)$ be an isomorphic embedding. If $n, k \in \en$, $n > k$, and $\alpha$ is an ordinal such that \[\omega^{\alpha} k < ht(K_2) \leq \omega^{\alpha}(k+1) \text{ and } ht(K_1) > \omega^{\alpha} n,\]
		then 
		\[\norm{T}\norm{T^{-1}} \geq \max\{3, \frac{2n-k}{k}\}.\]
		In particular, if $\C(K_1)$ is isomorphic to a subspace of $\C(K_2, E)$, then $\Gamma(ht(K_1)) \leq \Gamma(ht(K_2))$.
	\end{thm}
	
	The following corollary is immediate.
	
	\begin{cor}
		Let $K_1, K_2$ be infinite compact spaces, $E_1, E_2$ be  Banach spaces containing no copy of $c_0$. If $n, k \in \en$, $n > k$, and $\alpha$ is an ordinal such that \[\omega^{\alpha} k < ht(K_2) \leq \omega^{\alpha}(k+1) \text{ and } ht(K_1) > \omega^{\alpha} n,\]
		then 
		\[d_{BM}(\C(K_1, E_1), \C(K_2, E_2)) \geq \max\{3, \frac{2n-k}{k}\}.\]
		In particular, if $\C(K_1, E_1)$ is isomorphic to $\C(K_2, E_2)$, then $\Gamma(ht(K_1))=\Gamma(ht(K_2))$.
	\end{cor}
	
	We note that we will only need to prove that the lower bound $\frac{2n-k}{k}$ is true, since the bound $3$ is known, see \cite[Theorem 1.8]{GalegoVillamizar} or \cite[Theorem 1.1]{rondos-scattered-subspaces}. Also notice that the assumption that $c_0$ does not embed in $E_1, E_2$ cannot be plainly removed, which follows for example from the fact that
	
	\[\C([0, \omega], \C([0, \omega^{\omega}])) \simeq \C([0, \omega] \times [0, \omega^{\omega}]) \simeq \C([0, \omega^{\omega}], \C([0, \omega)]),\]
	
	but $\Gamma(ht([0,\omega]))=\Gamma(2)=\omega \neq \omega^2 =\Gamma(\omega+1)=\Gamma(ht([0,\omega^{\omega}]))$.
	
	Further, notice that Theorem \ref{theresult} together with the classical isomorphic classification due to Bessaga, Pe\l czy\'{n}ski, and Milutin yields that for compact metric spaces $K_1, K_2$ and a Banach space $E$ not containing an isomorphic copy of $c_0$, $\C(K_1, E)$ is isomorphic to $\C(K_2, E)$ if and only if $\C(K_1)$ is isomorphic to $\C(K_2)$. This gives a strengthening of the results from \cite{GALEGO_vector_classification_factor} and \cite{GALEGO_Zahn_classification}.
	
	Finally, a typographical note: the symbol \qedsymbol\ denotes the end of a proof, while, in nested proofs, we use $\square$ for the end of the inner proof.
	
	\section{The proofs}
	
	We start with the following lemma, which is essentially known in a slightly different formulation (see e.g. \cite[Proposition 2.3]{GalegoVillamizar} or \cite[Lemma 2.1]{rondos-scattered-subspaces}). Even though the proof is not complicated, we decided to include it for the convenience of the reader.
	
	\begin{lemma}
		\label{fin}
		Let $K_1, K_2$ be compact spaces, $E$ be a Banach space not containing an isomorphic copy of $c_0$, $L_1 \subseteq K_1$ be an infinite set, $U \supseteq L_1$ be an open set, $L_2 \subseteq K_2$ be a finite set and let $T:\C(K_1) \rightarrow \C(K_2, E)$ be an isomorphic embedding. Then for each $\ep>0$ there exist a function $f \in \C(K_1, [0, 1])$ and $x \in L_1$ such that $f=1$ on an open neighbourhood of $x$, $f=0$ on $K_1 \setminus U$, and such that $\norm{Tf(y)}<\ep$ for each $y \in L_2$.  
	\end{lemma}
	
	\begin{proof}
		Since $L_1$ is infinite, we can find pairwise distinct points $\{x_n\}_{n=1}^{\infty}$ in $L_1$ together with pairwise disjoint open sets $\{U_n\}_{n=1}^{\infty}$ contained in $U$, each $U_n$ containing $x_n$. Next we find functions $\{f_n\}_{n=1}^{\infty} \subseteq \C(K_1, [0, 1])$ such that for each $n$, $f_n=1$ on some open neighbourhood of $x_n$ and $f_n(K_1 \setminus U_n)=\{0\}$.
		
		Now, we claim that one of the functions $f_n$ has the desired properties. Assuming the contrary, there exists $\ep>0$ such that for each $n \in \en$, there exists $y\in L_2$ such that $\norm{Tf_n(y)}\geq \ep$. Then, since $L_2$ is finite, passing to a subsequence we may fix a point $y_0 \in L_2$ satisfying that $\norm{Tf_n(y_0)}\geq \ep$ for each $n \in \en$. 
		
		Now, by the classical characterization of the Banach spaces containing $c_0$, see \cite[Theorem 6.7]{morrison2001functional}, to finish the proof it is enough to show that the series $\sum_{n=1}^{\infty} Tf_n(y_0)$ is weakly unconditionally Cauchy in $E$, meaning that $\sum_{n=1}^{\infty} \abs{\la e^*, Tf_n(y_0) \ra}< \infty$ for each $e^* \in \E$.
		To show this, we consider the evaluation mapping $\phi\colon K_2 \times \E \to \C(K_2, E)^*$ defined as
		\[ \la \phi(y, \e), g \ra=\la \e, g(y) \ra, \quad g \in \C(K_2, E), y \in K_2, \e \in \E.\] 
		Clearly, $\norm{\phi(y, \e)}=\norm{\e}$. Now, we fix $e^* \in E^*$, and let $T^*$ be the adjoint of $T$. Further, for a fixed $n \in \en$, let $\alpha_1, \ldots, \alpha_n \in S_{\er}$ satisfy 
		\[\abs{\la T^*\phi(y_0, e^*), f_i \ra}=\alpha_i \la T^*\phi(y_0, e^*), f_i \ra, \quad i=1, \ldots, n.\]
		Then we have 
		\begin{equation}
		\nonumber
		\begin{aligned}
		&\sum_{i=1}^{n} \abs{\la e^*, Tf_i(y_0) \ra}=\sum_{i=1}^{n} \abs{\la \phi(y_0, e^*), Tf_i \ra}=
		\sum_{i=1}^{n} \abs{\la T^*\phi(y_0, e^*), f_i \ra}
		=\\&=
		\sum_{i=1}^{n} \alpha_i \la T^*\phi(y_0, e^*), f_i \ra=\la T^*\phi(y_0, e^*), \sum_{i=1}^{n} \alpha_i f_i \ra 
		\leq \\& \leq
		\norm{T^*\phi(y_0, e^*)} \norm{\sum_{i=1}^n \alpha_i f_i}\leq \norm{T^*} \norm{e^*}{\norm{\sum_{i=1}^n \alpha_i f_i}} = \norm{T^*}\norm{e^*}. 
		\end{aligned}
		\end{equation}
		Thus also $\sum_{n=1}^{\infty} \abs{\la e^*, Tf_n(y_0) \ra} \leq \norm{T^*}\norm{e^*}<\infty$, which finishes the proof.
	\end{proof}
	
	Further, the following important lemma is based on an idea of Gordon \cite[Lemma 2.2]{Gordon3}, which has been improved and generalized subsequently by several authors (\cite{CandidoGalego3}, \cite{CandidoGalegoComega}, \cite[Lemma 2.1]{CANDIDOc0} and \cite[Lemma 4.3]{rondos-scattered-subspaces}). 
	
	\begin{lemma}
		\label{norm}
		Let $K_1, K_2$ be compact spaces, $E$ be a Banach space, let $T:\C(K_1) \rightarrow \C(K_2, E)$ be an isomorphic embedding and let $L_2$ be a subset of $ K_2$. Let $n, k \in \en$, $n>k$, and let $\ep>0$ be given. Suppose that there exist functions $g_1, \ldots, g_n \in \C(K_1, [0, 1])$ and $x \in K_1$ such  that $g_1(x)=\ldots =g_n(x)=1$,
		and such that for each $y \in L_2$, the set
		\[\{ i \in \{1, \ldots, n\}: \norm{Tg_i(y)} \geq \ep \}\]
		has cardinality at most $k$. Then:
		\begin{itemize}
			\item[(a)] there exists a linear combination $h$ of the functions $g_1, \ldots, g_n$ such that $\norm{h}=1$ and $\norm{Th|_{L_2}} \leq  \frac{k\norm{T}+(n-k)\ep}{n}$.
			
			\item[(b)]If moreover $L_2=K_2$ and 
			$g_1 \leq \ldots \leq g_n$,
			then there exists a linear combination $f$ of the functions $g_1, \ldots, g_n$ such that $\norm{f}=1$ and
			\[\norm{Tf} \geq \frac{2n-k}{k\norm{T^{-1}}}-\frac{2(n-k)}{k}\ep.\]
		\end{itemize}
	\end{lemma}
	
	\begin{proof}
		The function $h$ is defined simply as $h=\frac{1}{n}\sum_{i=1}^n g_i$. Then $1 \geq \norm{h} \geq h(x)=\frac{1}{n}n=1$.
		If $y \in L_2$ is arbitrary, then there exist $n-k$ functions $g_{i_1}, \ldots , g_{i_{n-k}}$ satisfying that for each $j=1, \ldots ,n-k$, $\norm{Tg_{i_j}(y)}<\ep$. Thus 
		\[\norm{Th(y)} \leq \frac{1}{n}\sum_{i=1}^n \norm{Tg_i(y)}\leq \frac{k\norm{T}+(n-k)\ep}{n}.\]
		
		Now, suppose that $L_2=K_2$. To find the function $f$ we first consider the function $g=2\sum_{i=1}^{n-k} g_i+\sum_{i=n-k+1}^n g_i \in \C(K_1)$. We have 
		\[\norm{g} \geq g(x)=2(n-k)+k=2n-k.\]
		Thus \[\norm{Tg} \geq \frac{1}{\norm{T^{-1}}}\norm{g}\geq \frac{2n-k}{\norm{T^{-1}}},\]
		and hence there exists $y \in K_2$ such that $\norm{Tg(y)}\geq \frac{2n-k}{\norm{T^{-1}}}$.
		Next, there exist indices $i_1, \ldots, i_{n-k} \in \{1, \ldots n\}$ such that for each $j=1, \ldots, n-k$, $\norm{Tg_{i_j}(y)}<\ep$.
		We denote \[f=\frac{1}{k}(2\sum_{i=1}^{n-k} g_i+\sum_{i=n-k+1}^n g_i-2\sum_{j=1}^{n-k} g_{i_j})\]
		and we check that this function has the desired properties.
		Firstly, since $0 \leq g_1 \leq \ldots \leq g_n \leq 1$, we have
		
		\begin{equation}
		\nonumber
		\begin{aligned}
		-k&\leq -\sum_{i=n-k+1}^{n} g_{i}
		\leq 
		2\sum_{i=1}^{n} g_i-2\sum_{j=1}^{n-k}g_{i_j}-\sum_{i=n-k+1}^{n} g_{i}=\\&
		=2\sum_{i=1}^{n-k} g_i+\sum_{i=n-k+1}^{n} g_{i}-2\sum_{j=1}^{n-k}g_{i_j}=k f\leq
		\sum_{i=n-k+1}^{n} g_{i} \leq k.
		\end{aligned}
		\end{equation}
		Thus $\norm{f} \leq 1$. Moreover, 
		\[f(x)=\frac{1}{k}(2(n-k)+k-2(n-k))=1,\]
		and hence $\norm{f}=1$. On the other hand,
		\begin{equation}
		\nonumber
		\begin{aligned}
		\norm{Tf} &\geq \norm{Tf(y)} \geq \frac{1}{k}(\norm{Tg(y)}-2\sum_{j=1}^{n-k} \norm{Tg_{i_j}(y)}) \geq \frac{2n-k}{k\norm{T^{-1}}}-\frac{2(n-k)}{k} \ep,
		\end{aligned}
		\end{equation}
		which finishes the proof.
	\end{proof}

	Next, we are going to state some elementary results on scattered  derivatives.

	\begin{lemma}\label{l: maxscatt and open reduction}
		Let $K$ be a compact space. 
		\begin{itemize}
			\item[(a)] If $L_1,L_2$ are two closed scattered subspaces of $K$, then $ht(L_1\cup L_2)\leq \max\{ht(L_1),ht(L_2)\}$.
			\item[(b)] If $L$ is a subset of $K$ and $V$ is an open subset of $K$, then for each ordinal $\alpha$, $V \cap L^{(\alpha)} \subseteq (V \cap L)^{(\alpha)}$.
			\item[(c)] If $L$ is a subset of $K$ and $\alpha$ and $\beta$ are ordinal numbers, then it holds $(L^{(\alpha)})^{(\beta)}=L^{(\alpha+\beta)}$.
		\end{itemize}
	\end{lemma}
	
	\begin{proof}
		In order to get the assertion (a), it is enough to prove that for each ordinal $\alpha$, $(L_1 \cup L_2)^{(\alpha)} \subseteq L_1^{(\alpha)} \cup L_2^{(\alpha)}$. Let us prove it by using a transfinite induction argument. The case $\alpha=0$ is trivial. Next we consider the case $\alpha=1$. Thus we suppose that $x$ is an accumulation point of $L_1 \cup L_2$. We find a net $(x_{\lambda})_{\lambda \in \Lambda} \subseteq L_1 \cup L_2$ converging to $x$. Passing to a subnet, we may suppose that $(x_{\lambda})_{\lambda \in \Lambda} \cup \{x\} \subseteq L_1$. Thus $x \in L_1^{(1)}$, which finishes the proof for $\alpha=1$. 
		Now, we suppose that the statement holds for an ordinal $\alpha \geq 1$. Then we have
		\[(L_1 \cup L_2)^{(\alpha+1)}=((L_1 \cup L_2)^{(\alpha)})^{(1)} \subseteq (L_1^{(\alpha)} \cup L_2^{(\alpha)})^{(1)} \subseteq L_1^{(\alpha+1)} \cup L_2^{(\alpha+1)}.\]
		Further, for a limit ordinal $\alpha$ we have
		\[(L_1 \cup L_2)^{(\alpha)}=\bigcap_{\beta < \alpha} (L_1 \cup L_2)^{(\beta)} \subseteq \bigcap_{\beta < \alpha} (L_1^{(\beta)} \cup L_2^{(\beta)})=L_1^{(\alpha)} \cup L_2^{(\alpha)},\]
		which finishes the proof of (a).
		
		For the proof of (b) we proceed again by transfinite induction. The case $\alpha=0$ holds trivially. Next we prove the statement for $\alpha=1$. Thus we pick $x \in V \cap L^{(1)}$. Hence $x \in V \cap L$ and there exists a net $(x_{\lambda})_{\lambda \in \Lambda} \subseteq L$ converging to $x$. Then  $(x_{\lambda})_{\lambda \in \Lambda} \subseteq V$ eventually, and hence $x$ is an accumulation point of $V \cap L$. Now, we suppose that the statement holds for an ordinal $\alpha \geq 1$, and we get
		\[V \cap L^{(\alpha+1)} = V \cap (L^{(\alpha)})^{(1)} \subseteq (V \cap L^{(\alpha)})^{(1)} \subseteq ((V \cap L)^{(\alpha)})^{(1)}=(V \cap L)^{(\alpha+1)}.
		\]
		Finally, for a limit ordinal $\alpha$ we have 
		\[V \cap L^{(\alpha)}=V \cap \bigcap_{\beta<\alpha} L^{(\beta)}=\bigcap_{\beta<\alpha} V \cap L^{(\beta)} \subseteq \bigcap_{\beta<\alpha} (V \cap L)^{(\beta)}=(V \cap L)^{(\alpha)}.\]
		Finally, we are going to prove (c). Let us prove it by induction on $\beta$. For $\beta=0$ the formula holds trivially. Suppose that for each $\gamma<\beta$ it holds $(L^{(\alpha)})^{(\gamma)}=L^{\alpha+\gamma}$, hence if $\beta=\delta+1$ we have, by definition of derivative, $(L^{(\alpha)})^{(\beta)}=((L^{(\alpha)})^{(\delta)})^{(1)}=(L^{(\alpha+\delta)})^{(1)}=L^{(\alpha+\delta+1)}=L^{(\alpha+\beta)}$. While if $\beta$ is a limit ordinal, we get $(L^{(\alpha)})^{(\beta)}=\bigcap_{\gamma<\beta} (L^{(\alpha)})^{(\gamma)}=\bigcap_{\gamma<\beta} L^{(\alpha+\gamma)}=\bigcap_{\gamma<\alpha+\beta} L^{(\gamma)}=L^{(\alpha+\beta)}$.
		The proof is finished.
	\end{proof}

	The following proposition, which essentially contains the proof of Theorem \ref{theresult}, is inspired by the approach of Bessaga and Pe\l czy\'{n}ski \cite{BessagaPelcynski_classification} and by \cite[Theorem 1.3]{rondos-scattered-subspaces}. Given a set $\Gamma$ and a positive integer $n$, we denote $[\Gamma]^n:=\{A\subset \Gamma: |A|=n\}$.
	\begin{prop}
		\label{main}
		Let $K_1, K_2$ be nonempty compact spaces and $E$ be a Banach space not containing an isomorphic copy of $c_0$. Suppose that $T:\C(K_1) \rightarrow \C(K_2, E)$ is an isomorphic embedding, let $L_1 \subseteq K_1$, $U$ be an open set containing $L_1$, and let $L_2 \subseteq K_2$ be a compact set. Then the following assertions hold.
		\begin{itemize}
			\item[(a)] If $\Gamma(ht(L_1))>\Gamma(ht(L_2))$, then for each $\ep>0$ there exist a function $f \in \C(K_1, [0, 1])$ and $x \in L_1$ such that $f=1$ on an open neighbourhood of $x$, $f=0$ on $K_1 \setminus U$ and $\norm{Tf|_{L_2}}<\ep$.
			\item[(b)] If $n, k \in \en$, $n > k$, and $\alpha$ is an ordinal such that 
			\[\omega^{\alpha} k < ht(L_2) \leq \omega^{\alpha}(k+1) \text{ and } ht(L_1) > \omega^{\alpha} n,\]
			then for each $\ep>0$ there exist a function $h \in \C(K_1, [0, 1])$ and a point $x \in L_1$ such that $h=1$ on an open neighbourhood of $x$, $h=0$ on $K_1 \setminus U$ and 
			\[\norm{Th|_{L_2}} \leq \frac{k\norm{T}+(n-k)\ep}{n}.\]
			\item[(c)] If $n, k \in \en$, $n > k$, and $\alpha$ is an ordinal such that 
			\[\omega^{\alpha} k < ht(K_2) \leq \omega^{\alpha}(k+1) \text{ and } ht(K_1) > \omega^{\alpha} n,\]
			then for each $\ep>0$ there exist a function $f \in \C(K_1, [0, 1])$ of norm $1$ such that 
			\[\norm{Tf} \geq \frac{2n-k}{k\norm{T^{-1}}}-\frac{2(n-k)}{k}\ep.\] 
		\end{itemize}
	\end{prop}
	
	\begin{proof}
		
		Throughout the proof, for a function $g \in \C(K_1)$ and $\ep>0$ we will use $\{ \norm{Tg} \geq \ep \}$ as a shortcut for $\{ y \in K_2: \norm{Tg(y)} \geq \ep \}$.
		
		Further, for the sake of clarity we denote by $A(\alpha), B(\alpha)$ and $C(\alpha)$ the statements that the assertions (a), (b) and (c) respectively hold for each compact set $L_2 \subseteq K_2$ with $\Gamma(ht(L_2))\leq \omega^{\alpha}$ (or for each compact space $K_2$ satisfying $\Gamma(ht(K_2))\leq \omega^{\alpha}$ in the case of (c)). We proceed to prove simultaneously (a), (b) and (c) by transfinite induction on $\alpha$ based on the following scheme (note that the assertions $B(0)$ and $C(0)$ hold trivially, since if $\Gamma(ht(L_2)) \leq \omega^0=1$, then $ht(L_2)=1$, and hence the inequality $\omega^{\alpha} k < ht(L_2)$ is never satisfied).
		\begin{equation}
		\nonumber
		\begin{aligned}
		&A(0) \text{ holds}, \\&
		A(\alpha) \wedge B(\alpha) \Rightarrow B(\alpha+1),\\&
		A(\alpha) \wedge C(\alpha) \Rightarrow C(\alpha+1),\\&
		B(\alpha+1) \wedge A(\alpha) \Rightarrow A(\alpha+1), \\&
		A(\beta) \text{ for each } \beta <\alpha \Rightarrow A(\alpha), \text{ for } \alpha \text{ limit}, \\&
		B(\beta) \text{ for each } \beta <\alpha \Rightarrow B(\alpha), \text{ for } \alpha \text{ limit}, \text{ and }\\&
		C(\beta) \text{ for each } \beta <\alpha \Rightarrow C(\alpha), \text{ for } \alpha \text{ limit}.
		\end{aligned}
		\end{equation}
		
		To start with, let $L_2 \subseteq K_2$ be a a compact set with $\Gamma(ht(L_2))=1$. Thus $ht(L_2) \leq 1$, and hence, since $L_2$ is compact, $L_2$ is a finite set. If $\Gamma(ht(L_1))>1$, then $ht(L_1)>1$, and hence $L_1$ is infinite. Thus the assertion $A(0)$ follows by Lemma \ref{fin}.
		
		Now, assume that $A(\alpha)$ and $B(\alpha)$ hold. We want to prove $B(\alpha+1)$. Thus we pick a compact set $L_2 \subseteq K_2$ satisfying that $\Gamma(ht(L_2))=\omega^{\alpha+1}$. This means that there exists $k \in \en$ such that $\omega^{\alpha} k < ht(L_2) \leq \omega^{\alpha}(k+1)$. Let a set $L_1 \subseteq K_1$ satisfies $ht(L_1) > \omega^{\alpha} n$ for some $n>k$, $U$ be an open set containing $L_1$, and fix $\ep>0$. 
		
		\emph{Claim.
			For each $i=1, \ldots, n$ there exist functions
			$g_1, \ldots g_i$ in $\C(K_1, [0, 1])$ and points $x_1, \ldots, x_i \in K_1$
			such that
			\[1 \geq g_1 \geq g_2 \geq \ldots \geq g_i \geq 0,\] 
			and for each $j=1, \ldots, i$, $x_j \in L_1^{(\omega^{\alpha}(n-j))}$, $g_j=1$ on an open neighbourhood of $x_j$ and $g_j=0$ on $K_1 \setminus U$, and such that
			\begin{equation}
			\label{empty}
			\bigcup_{s=1}^{i} (L_2^{(\omega^{\alpha} \max\{k-s+1, 0\})} \cap \bigcup_{A\in[\{1,\dots,i\}]^s} \bigcap_{p \in A} \{ \norm{Tg_p} \geq \ep\})=\emptyset. 
			\end{equation}}
		
		\begin{proof}[Proof of Claim] \renewcommand\qedsymbol{$\square$}
			We proceed by finite induction. 
			
			To start with, we have $ht(L_1^{(\omega^{\alpha}(n-1))})> \omega^{\alpha}$, $ht(L_2^{(\omega^{\alpha} k)})\leq \omega^{\alpha}$. Thus \\ $\Gamma(ht(L_1^{(\omega^{\alpha}(n-1))}))> \omega^{\alpha}$ and $\Gamma(ht(L_2^{(\omega^{\alpha} k)})) \leq \omega^{\alpha}$. Hence by $A(\alpha)$, there exist a function $g_1 \in \C(K_1, [0, 1])$ and $x_1 \in L_1^{(\omega^{\alpha}(n-1))}$ such that $g_1=1$ on an open neighbourhood of $x_1$, $g_1=0$ on $K_1\setminus U$ and $\norm{Tg_1(y)}<\ep$ for each $y \in L_2^{(\omega^{\alpha} k)}$.
			
			Now, suppose that $1 \leq i<n$, and that we have found the functions $g_1, \ldots g_i$ in $\C(K_1, [0, 1])$ and points $x_1, \ldots, x_i \in K_1$
			such that \eqref{empty} holds for $i$, and satisfying all the other above conditions. 
			Hence we know that the set 
			
			\[	\bigcup_{s=1}^{i} (L_2^{(\omega^{\alpha} \max\{k-s+1, 0\})} \cap \bigcup_{A\in[\{1,\dots,i\}]^s} \bigcap_{p \in A} \{ \norm{Tg_p} \geq \ep\}) \]
			is empty. Thus for each $s=1, \ldots, i$, 
			\[ ht(L_2^{(\omega^{\alpha} \max\{k-s, 0\})} \cap \bigcup_{A\in[\{1,\dots,i\}]^s} \bigcap_{p \in A} \{ \norm{Tg_p} \geq \ep\})\leq\omega^{\alpha}.\]
			
			Thus, since $ht(L_2^{(\omega^{\alpha} k)})\leq \omega^{\alpha}$ by the assumption, if we denote 
			\[M=L_2^{(\omega^{\alpha} k)} \cup \bigcup_{s=1}^{i} (L_2^{(\omega^{\alpha} \max\{k-s, 0\})} \cap \bigcup_{A\in[\{1,\dots,i\}]^s} \bigcap_{p \in A} \{ \norm{Tg_p} \geq \ep\}),\] 
			then, by Lemma \ref{l: maxscatt and open reduction}(a),
			also $ht(M)\leq \omega^{\alpha}$. Notice also that the set $M$ is compact.
			
			Next, we know that we can find an open neighbourhood $V$ of $x_i$ contained in $U$ such that $g_i=1$ on $V$. 
			Since $x_i \in V \cap L_1^{(\omega^{\alpha}(n-i))}$, using Lemma \ref{l: maxscatt and open reduction}(b) and (c) on the sets $L_1^{(\omega^{\alpha}(n-i-1))}$ and $V$  we deduce that $ht(V \cap L_1^{(\omega^{\alpha}(n-i-1))}) > \omega^{\alpha}$.
			Thus by $A(\alpha)$, there exist a function $g_{i+1}$ and a point $x_{i+1} \in V \cap L_1^{(\omega^{\alpha}(n-i-1))}$ such that $g_{i+1}=1$ on an open neighbourhood of $x_{i+1}$, $g_{i+1}=0$ on $K_1 \setminus V$ and $\norm{Tg_{i+1}(y)}<\ep$ for each $y \in M$. Then $g_{i+1}\leq g_i$. Thus to finish the proof of the claim, it is now enough to check that \eqref{empty} holds for $g_1, \ldots, g_{i+1}$. To this end, we have 
			\begin{equation}
			\nonumber
			\begin{aligned}
			\emptyset&=\{\norm{Tg_{i+1}} \geq \ep\} \cap M=
			\{\norm{Tg_{i+1}} \geq \ep\} \cap (L_2^{(\omega^{\alpha} k)} \cup \\& \bigcup_{s=1}^{i} (L_2^{(\omega^{\alpha} \max\{k-s, 0\})} \cap \bigcup_{A\in[\{1,\dots,i\}]^s} \bigcap_{p \in A} \{ \norm{Tg_p} \geq \ep\}))
			=\\&=
			(\{\norm{Tg_{i+1}} \geq \ep\} \cap L_2^{(\omega^{\alpha} k)}) \cup \\& \bigcup_{s=1}^{i} (L_2^{(\omega^{\alpha} \max\{k-s, 0\})} \cap \bigcup_{A\in[\{1,\dots,i\}]^s} \bigcap_{p \in A} (\{ \norm{Tg_p} \geq \ep\} \cap \{\norm{Tg_{i+1}} \geq \ep\}))
			=
			\\&=
			(\{\norm{Tg_{i+1}} \geq \ep\} \cap L_2^{(\omega^{\alpha} k)}) \cup \\& \bigcup_{s=1}^{i} (L_2^{(\omega^{\alpha} \max\{k-s, 0\})} \cap \bigcup_{A\in[\{1,\dots,i\}]^{s}} \bigcap_{p \in A\cup\{i+1\}} \{ \norm{Tg_p} \geq \ep\})
			=\\&
			=
			\bigcup_{s=0}^{i} (L_2^{(\omega^{\alpha} \max\{k-s, 0\})} \cap \bigcup_{A\in[\{1,\dots,i\}]^{s}} \bigcap_{p \in A\cup\{i+1\}} \{ \norm{Tg_p} \geq \ep\})
			=\\&=
			\bigcup_{s=1}^{i+1} (L_2^{(\omega^{\alpha} \max\{k-s+1, 0\})} \cap \bigcup_{A\in[\{1,\dots,i\}]^{s-1}} \bigcap_{p \in A\cup\{i+1\}} \{ \norm{Tg_p} \geq \ep\}).
			\end{aligned}
			\end{equation}
			Thus recalling the inductive assumption we conclude that the set
			\begin{equation}
			\nonumber
			\begin{aligned}
			&\bigcup_{s=1}^{i+1} (L_2^{(\omega^{\alpha} \max\{k-s+1, 0\})} \cap \bigcup_{A\in[\{1,\dots,i+1\}]^{s}} \bigcap_{p \in A} \{ \norm{Tg_p} \geq \ep\})=\\&
			\bigcup_{s=1}^{i+1} (L_2^{(\omega^{\alpha} \max\{k-s+1, 0\})} \cap \bigcup_{A\in[\{1,\dots,i\}]^{s-1}} \bigcap_{p \in A\cup\{i+1\}} \{ \norm{Tg_p} \geq \ep\})
			\cup \\& \cup \bigcup_{s=1}^{i} (L_2^{(\omega^{\alpha} \max\{k-s+1, 0\})} \cap \bigcup_{A\in[\{1,\dots,i\}]^s} \bigcap_{p \in A} \{ \norm{Tg_p} \geq \ep\})
			\end{aligned}
			\end{equation}
			is empty. This finishes the induction step and the proof of the claim. 
		\end{proof}
		
		Now, we use the above claim applied for the case when $i=n$ to obtain the functions $g_1, \ldots, g_n$. Then from the formula \eqref{empty} in the special case when $s=k+1$ we obtain that the set 
		\[L_2 \cap \bigcup_{A\in[\{1,\dots n\}]^{k+1}} \bigcap_{p \in A} \{ \norm{Tg_p} \geq \ep\}\] is empty. Thus we can use Lemma \ref{norm}(a) to obtain a linear combination $h$ of the functions  $g_1, \ldots, g_n$ which satisfies $\norm{h}=1$ and \[\norm{Th|_{L_2}} \leq \frac{k\norm{T}+(n-k)\ep}{n}.\]
		Moreover, since each of the functions $g_i$ is constant $1$ on an open neighbourhood of the point $x_n \in L_1^{(0)}=L_1$, so is $h$. Finally, since each of the functions $g_i$ satisfies $g_i=0$ on $K_1 \setminus U$, so does the function $h$, which finishes the proof of $B(\alpha+1)$.
		
		Moreover, in the case when $L_1=K_1$ and $L_2=K_2$, if we assume that $C(\alpha)$ holds instead of $B(\alpha)$, an application of Lemma \ref{norm}(b) proves $C(\alpha+1)$.
		
		Next, suppose that $B(\alpha+1)$ and $A(\alpha)$ hold. We pick a compact set $L_2 \subseteq K_2$ satisfying $\Gamma(ht(L_2))=\omega^{\alpha+1}$, and an arbitrary set $L_1 \subseteq K_1$ satisfying $\Gamma(ht(L_1))>\Gamma(ht(L_2))$. Let $U$ be an open set containing $L_1$ and let $\ep>0$ be arbitrary. We find $k \in \en$ such that 
		$\omega^{\alpha} k < ht(L_2) \leq \omega^{\alpha}(k+1)$. Further, since $\Gamma(ht(L_1))>\Gamma(ht(L_2))=\omega^{\alpha+1}$, it follows that $ht(L_1)>\omega^{\alpha+1}$. Hence for each $n>k$, $ht(L_1)>\omega^{\alpha}n$, in particular for $n$ satisfying that \[\frac{k\norm{T}+(n-k)\frac{\ep}{2}}{n} \leq \ep.\]
		
		Hence by $B(\alpha+1)$, there exists a function $h \in \C(K_1, [0, 1])$ of norm $1$ and a point $x \in L_1$ such that $h=1$ on an open neighbourhood of $x$, $h=0$ on $K_1 \setminus U$ and \[\norm{Th|_{L_2}} \leq \frac{k\norm{T}+(n-k)\frac{\ep}{2}}{n} \leq \ep,\]
		which proves $A(\alpha+1)$.
		
		Finally, notice that the limit steps are trivial. This follows from the fact that, since $ht(L_2)$ is a successor ordinal for each compact set $L_2 \subseteq K_2$, if $\Gamma(ht(L_2)) \leq \omega^\alpha$ for some limit ordinal $\alpha$, then there exists $\beta<\alpha$ such that $\Gamma(ht(L_2)) \leq \omega^\beta$. The proof is finished. \end{proof}
	
	The proof of our main result now follows promptly from Proposition \ref{main}.
	
	\begin{proof}[Proof of  Theorem \ref{theresult}.] We recall that the lower bound $3$ is known, see e.g. \cite[Theorem 1.1]{rondos-scattered-subspaces}. The bound $\frac{2n-k}{k}$ follows immediately from Proposition \ref{main}(c), and the "in particular" statement can be easily deduced from this estimate, or alternatively, if follows directly from Proposition \ref{main}(a). \end{proof}

	
	\bibliography{bib}\bibliographystyle{siam}
\end{document}